\documentclass{amsart}

\usepackage[dvips]{graphicx}
\usepackage{amssymb}
\usepackage{mdwlist}
\usepackage{hhline}

\newtheorem{theorem}{Theorem}[section]

\newtheorem{proposition}[theorem]{Proposition}
\newtheorem{fact}[theorem]{Fact}
\newtheorem{problem}[theorem]{Problem}

\theoremstyle{definition}
\newtheorem{definition}[theorem]{Definition}

\theoremstyle{remark}

\begin{document}

\title{Surgical distance between lens spaces}

\author{Kazuhiro Ichihara} 
%\thanks{The first author is supported in part by Grant-in-Aid for Young Scientists (B), No.20740039, 
%The Ministry of Education, Culture, Sports, Science and Technology, Japan. }
\address{School of Mathematics Education, 
Nara University of Education, 
Takabatake-cho, Nara 630--8528, Japan.} 
\email{ichihara@nara-edu.ac.jp}

\author{Toshio Saito} 
%\thanks{The second author is partially supported by 
%JSPS Research Fellowships for Young Scientists.}
\address{Graduate School of Humanities and Sciences,
Nara Women's University,
Kitauoyanishi-machi, Nara 630--8506, Japan.} 
\email{tsaito@cc.nara-wu.ac.jp} 

%===============================================


%===============================================

\bigskip

\dedicatory{Dedicated to Professor Akio Kawauchi on the occasion of his 60th birthday}

%===============================================

\date{\today} 

\begin{abstract}
It is well-known that any pair of closed orientable 3-manifolds
are related by a finite sequence of Dehn surgeries on knots.
Furthermore Kawauchi showed that such knots can be taken to be hyperbolic.
In this article, 
we consider the minimal length of such sequences connecting
a pair of 3-manifolds, in particular, a pair of lens spaces.
\end{abstract}

\maketitle

%%%%%%%%%%%%%%%%%%%%%%%%%%%%%%%%%%%%%%%%%%%%%
\section{Introduction}

As a consequence of 
the famous Geometrization Conjecture raised by W.P. Thurston in \cite[section 6, question 1]{Th2}, 
all closed orientable $3$-manifolds are classified as follows: 
They should be; 
reducible (i.e., containing essential 2-spheres), 
toroidal (i.e., containing essential tori), 
Seifert fibered (i.e., foliated by circles), 
or hyperbolic manifolds 
(i.e., admitting a complete Riemannian metric with constant sectional curvature $-1$). 
Also see \cite[Problem 3.45]{Kirby}, and see \cite{Scott} for a survey. 

Now, by the celebrated Perelman's works \cite{P1, P2, P3}, 
an affirmative answer to this Geometrization Conjecture could be given. 
Beyond the classification, 
one of the next directions in the study of 3-manifolds 
would be to consider relationships between $3$-manifolds. 
One of the important operations describing such relationships 
must be \textit{Dehn surgery}. 
This is an operation to create a new $3$-manifold 
from a given one and a given knot 
(i.e., an embedded simple closed curve) 
in the following way: 
Remove an open tubular neighborhood of the knot, and glue a solid torus back. 
It gives an interesting subject to study; 
because, for instance, 
it is known that any pair of connected closed orientable $3$-manifolds 
are related by a finite sequence of Dehn surgeries on knots, 
proved by Lickorish \cite{Lick} and Wallace \cite{W} independently. 
See also Fact 1 below. 

In this article, in terms of Dehn surgery on knots, 
we introduce a \textit{distance} between pairs of $3$-manifolds.
Furthermore, by considering the surgery on \textit{hyperbolic knots}, 
another distance function is also defined, 
and we report the study of its restriction on the set of lens spaces. 

Throughout the article, for convenience, we denote 
by $\mathcal{M}$ 
the set of orientation preserving homeomorphism types 
of connected closed orientable $3$-manifolds.

%%%%%%%%%%%%%%%%%%%%%%%%%%%%%%%%%%%%%%%%%%%%%
\section{Backgrounds}

In this section, we will introduce some new definitions about Dehn surgery, 
and review backgrounds and known results about them. 
Also we will state a number of open problems which we will consider.

%%%%%%%%%%%%%%%%%%%%%%%
\subsection{Surgical distance}

First of all, we introduce 
a function $d : \mathcal{M} \times \mathcal{M} \to \mathbb{Z}_{\ge 0}$ 
defined as follows; 
for $  [ M ] ,  [ M' ]  \in \mathcal{M}$, 
$d( [ M ] , [ M' ] )$ is defined as the minimal length of the sequence 
$ [ M ] = [ M_0 ] , [ M_1 ] , \cdots, [ M_n ] = [ M' ] \in \mathcal{M}$ 
such that $ M_{i+1} $ is obtained from $ M_i  $ by Dehn surgery on a knot.

\smallskip

It is easy to verify that if the function $d$ is well-defined, 
then it satisfies the axiom of distance function. 
Further, as we cited above, the following is known: 

\begin{fact}[{Lickorish \cite{Lick}, Wallace \cite{W}}]
The function $d : \mathcal{M} \times \mathcal{M} \to \mathbb{Z}_{\ge 0}$ is well-defined. 
That is, 
for any pair $ [ M ] ,  [ M' ]  \in \mathcal{M}$, 
there exists a finite sequence 
$ [ M ] = [ M_0 ] , [ M_1 ] , \cdots, [ M_n ] = [ M' ] \in \mathcal{M}$ 
such that $ M_{i+1} $ is obtained from $ M_i  $ by Dehn surgery on a knot.
\end{fact}

Remark that, in \cite{Auckly}, 
Auckly defined a similar notion; ``surgery number" of $ [ M ] \in \mathcal{M}$. 
This is equal to $d( [ S^3 ] , [ M ] )$ in our definition. 
Also see \cite[Problem 3.102]{Kirby}. 

Also remark that, by a Dehn surgery on a knot, 
the first betti number $\beta_1$ of a $3$-manifold 
can be changed only by $\pm 1$. 
So 
$d ([ M ] ,  [ M' ] ) \ge | \beta_1 (M) - \beta_1 ( M' )|$ holds 
for $ [ M ] ,  [ M' ]  \in \mathcal{M}$. 
Thus it would be natural to ask: 

\begin{problem}
For any given $N >0$, 
can we find a pair $ [ M ] ,  [ M' ]  \in\mathcal{M}$ such that 
$\beta_1 (M) = \beta_1 ( M' )$ but $d ( [ M ] ,  [ M' ] ) \ge N$? 
\end{problem}

Here we collect several related known facts: 

\begin{itemize}

\item
All lens spaces have the first betti number at most one. 
And $d ( [ L ] , [ L' ] ) = 1$ for any lens spaces $L, L'$. 
See later for definitions of lens spaces. 

\item
In \cite[Theorem 3]{GL}, 
Gordon and Luecke showed $d ( [ S^3 ] , [ M ] ) > 1$ 
if $M$ is non-prime with lens space summands. 
Thus we can obtain infinitely many $3$-manifolds $M$ 
with $\beta_1 (M) = \beta_1 (S^3) = 0$ with $d ( [ S^3 ] ,  [ M ] ) > 1$. 

\item
In \cite{Auckly}, Auckly found the first hyperbolic example 
$ [ M ] \in \mathcal{M}$ with $\beta_1 ( M ) =0$ 
such that $d ( [ S^3 ] , [ M ] ) > 1$. %(Auckly)

\end{itemize}

As far as the authors know, 
there are no explicit examples of pairs of manifolds 
for which the surgical distance is determined to be three or more.

%%%%%%%%%%%%%%%%%%%%%%%%%%%
\subsection{Hyperbolic surgical distance}

Next we consider Dehn surgery on \textit{hyperbolic knots}, 
that is, the knots with complements 
which admits complete hyperbolic metric of finite volume. 
In fact, we introduce 
a function $d_H : \mathcal{M} \times \mathcal{M} \to \mathbb{Z}_{\ge 0}$ 
defined as follows; 
for $  [ M ] ,  [ M' ]  \in \mathcal{M}$, 
$d_H ( [ M ] , [ M' ] )$ is defined as the minimal length of the sequence 
$ [ M ] = [ M_0 ] , [ M_1 ] , \cdots, [ M_n ] = [ M' ] \in \mathcal{M}$ 
such that $ M_{i+1} $ is obtained from $ M_i  $ by Dehn surgery on a hyperbolic knot.

One reason why we choose to consider hyperbolic knots is as follows: 
Following the classification of $3$-manifolds, 
all knots are also classified into several types. 
When one considers only knots in types of hyperbolic, 
the next was established by Kawauchi using his ``Imitation Theory''. 
See \cite{Kawauchi} for example.

\begin{fact}[Kawauchi]
$d_H : \mathcal{M} \times \mathcal{M} \to \mathbb{Z}_{\ge 0}$ is well-defined. 
That is, 
for any pair $ [ M ] ,  [ M' ]  \in \mathcal{M}$, 
there exists a finite sequence 
$ [ M ] = [ M_0 ] , [ M_1 ] , \cdots, [ M_n ] = [ M' ] \in \mathcal{M}$ 
such that $M_{i+1}$ is obtained from $M_i$ by Dehn surgery on a hyperbolic knot.
\end{fact}

It is then easy to verify that this function also satisfies the axiom of distance function. 

Furthermore, Kawauchi showed the following: 

\begin{fact}[Kawauchi]
For $ [ M ] ,  [ M' ]  \in \mathcal{M}$, 
$$
d_H ( [ M ] ,  [ M' ] ) =
\left\{
\begin{array}{ll}
1 \quad \mathrm{ or } \quad 2	&	\qquad \mathrm{ if } \quad d( [ M ] , [ M' ] ) =1 \\
d( [ M ] , [ M' ] )				&	\qquad \mathrm{ otherwise }  \\
\end{array}
\right.
$$
\end{fact}

Then it seems to be natural to ask: 

\begin{problem}\label{problem2.5} 
When can $d ( [ M ] , [ M' ] ) \ne d_H ( [ M ] , [ M' ] )$ occur?
\end{problem}

Concerning this question, there are several known facts. 
We collect them in the following. 
%Note that they are not in the original forms. 

\begin{itemize}

\item
$d ([ S^3 ] , [ L(p,q) ] ) = 1 $ and $d_H ( [ S^3 ] , [ L(p,q) ] ) = 2$ 
if $q$ is not a quadratic residue modulo $p$; 
i.e., $ x^2 \not\equiv \pm q \mod p$ for any $x$. 
(Fintushel-Stern \cite[Proposition 1]{FS})

\item
$d_H ( [ S^3 ] , [ S^2 \times S^1 ] ) =2 $, 
while $d ( [ S^3 ] , [ S^2 \times S^1 ] ) = 1 $. (Gabai \cite{Gabai})

\item
There is a pair of lens spaces $L, L'$ 
such that  $d_H ( [ L ] , [ L' ] ) = 1$ and 
$L$ and $L'$ are orientation-reversingly homeomorphic. (Bleiler-Hodgson-Weeks \cite{BHW})
There is only one known example with such a property. 
See \cite[Problem 1.81]{Kirby} for related conjectures. 

\item
$d ( [ S^3 ] , [ L(p,q) ] ) = 1 $ and $d_H ( [ S^3 ] , [ L(p,q) ] ) = 2$ if $| p | < 9$.
In particular, 
$d ( [ S^3 ] , [ RP^3 ] ) = 1 $ and $d_H ( [ S^3 ] , [ RP^3 ] ) = 2$. 
(Kronheimer-Mrowka-Ozsv\'{a}th-Z. Szab\'{o} \cite[Theorem 1.1.]{KMOS}, Ozsv\'{a}th-Szab\'{o} \cite{OS})

\item
For the Poincar\'{e} homology sphere $P$, 
$d ( [ S^3 ] , [ P ] ) = 1 $ and $d_H ( [ S^3 ] , [ P ] ) = 2$. (Ghiggini \cite{Ghiggini})

\item
There is a sufficient condition to be $d_H ( [ S^2 \times S^1 ] , [ L ] ) = 2$ 
for a lens space $L$. (Lisca \cite{Lisca})

\end{itemize}

Please remark that 
the facts above could be obtained mainly 
from the results in the references together with many other facts.

%%%%%%%%%%%%%%%%%%%%%%%%%%%%%%%%%%%%%%%%%
\section{On the set of lens spaces}

In the rest of the article, we will concentrate on the set of lens spaces. 
We here call a 3-manifold $L$ with Heegaard genus at most one 
(i.e., constructed by gluing two solid tori) a \textit{lens space}. 
Thus, in this article, we say that $S^3$, $S^2 \times S^1$ and $RP^3$ are all lens spaces. 
Denote by $\mathcal{L}$ the set of orientation preserving homeomorphism types of lens spaces. 
Then note that 
$ d ( [ L ] , [ L' ] ) = 1 $ and $ d_H ( [ L ] , [ L' ] ) \le 2 $ 
for any $ [ L ] , [ L' ] \in \mathcal{L}$

\begin{definition}
For $ [ L ] , [ L' ] \in \mathcal{L}$, we set 
$d_H ( [ L ] , [ L' ] )_\mathcal{L} $ as 
the minimal length of the sequence 
$ [ L ] = [ L_0 ] , [ L_1 ] , \cdots, [ L_n ] = [ L' ] \in \mathcal{L}$ 
such that $ L_{i+1} $ is obtained from $ L_i  $ by Dehn surgery on a hyperbolic knot.
\end{definition}

\begin{problem}
Can $ d_H ( [ L ] , [ L' ] )_\mathcal{L} $ be well-defined for any $ [ L ] , [ L' ] \in \mathcal{L}$?
Equivalently, 
for any pair $ [ L ] , [ L' ] \in \mathcal{L}$, 
does there exist a finite sequence 
$ [ L ] = [ L_0 ] , [ L_1 ] , \cdots, [ L_n ] = [ L' ] \in \mathcal{L}$ 
such that $L_{i+1}$ is obtained from $L_i$ by Dehn surgery on a hyperbolic knot.
\end{problem}

Recall that: 
If $d_H ( [ L ] , [ L' ] ) =1$, then $ d_H ( [ L ] , [ L' ] )_\mathcal{L} = 1$ by definition. 
However $ d_H ( [ L ] , [ L' ] )_\mathcal{L} \ge d_H ( [ L ] , [ L'] ) =2$ in general. 

\begin{problem}\label{problem3.3}
Are there $[ L ] , [ L' ] \in \mathcal{L}$ such that 
$ d_H ( [ L ] , [ L' ] )_\mathcal{L} > d_H ( [ L ] , [ L' ] ) $?
Equivalently, 
are there $[ L ] , [ L' ] \in \mathcal{L}$ such that 
$ d_H ( [ L ] , [ L' ] )_\mathcal{L} > 2$?
\end{problem}

Usually, lens spaces are parametrized by a pair of coprime integers as follows. 
Let $V_1$ be a regular neighborhood of a trivial knot in $S^3$, 
$m$ a meridian of $V_1$ and $\ell$ a longitude of $V_1$ 
such that $\ell$ bounds a disk in $\mathrm{cl}(S^3\setminus V_1)$. 
We fix an orientation of $m$ and $\ell$ 
as illustrated in Figure \ref{FigA}. 
By attaching a solid torus $V_2$ to $V_1$ so that $\bar{m}$ is isotopic to 
a representative of  $p[\ell]+q[m]$ in $\partial V_1$, 
we obtain a lens space, 
which is denoted by $L(p,q)$, 
where $p$ and $q$ are integers with $p>0$ and $(p,q)=1$, 
and $\bar{m}$ is a meridian of $V_2$. 
It is known that two lens spaces $L(p,q)$ and $L(p',q')$ are 
(possibly orientation reversing) homeomorphic, 
i.e., $L(p,q) \cong L(p',q')$ if and only if 
$|p|=|p'|$, and $q\equiv \pm q'$ $(\bmod\ p)$ or $qq'\equiv \pm 1$ $(\bmod\ p)$.  
See \cite{R} for example.

%%%%%%%%%%%%%%%%%%%%%%%%%%%%%%%%%%%%%%%%%

\section{Results}

In this section, we will give our results concerning to Problems \ref{problem2.5} and \ref{problem3.3}. 

Recall that $d ( [ L ] , [ L' ] ) = 1 $ and $d_H ( [ L ] , [ L' ] ) \le  2$ for any $[ L ] , [ L' ] \in \mathcal{L}$. 
So consider the question: For which $[ L ] , [ L' ] \in \mathcal{L}$, $d_H ( [ L ] , [ L' ] ) = 1$? 

We here recall that basic terminology about Dehn surgery on knots in the $3$-sphere. 
See \cite{R} in details for example. 
As usual, by a \textit{slope}, we call an isotopy class of 
a non-trivial unoriented simple closed curve on a torus. 
Then Dehn surgery on a knot $K$ is characterized 
by the slope on the peripheral torus of $K$ 
which is represented by the simple closed curve 
identified with the meridian of the attached solid torus via the surgery. 
When $K$ is a knot in $S^3$, 
by using the standard meridian-longitude system, 
slopes on the peripheral torus are parametrized by 
rational numbers with $1/0$. 
For example, the meridian of $K$ corresponds to $1/0$ and the longitude to $0$. 
%
%By the \textit{trivial} Dehn surgery on $K$ in $S^3$, 
%we mean the Dehn surgery on $K$ along the meridional slope $1/0$. 
%Thus it yields $S^3$ again, which is obviously exceptional. 
%
We say that 
a Dehn surgery on $K$ in $S^3$ is $p/q$-surgery if it is along the slope $p/q$. 
This means that the curve representing the slope 
runs meridionally $p$ times and longitudinally $q$ times.

Let $V$ be  a solid torus standardly embedded in $S^3$, $K_1$ the closure of an $n$-string 
braid in $V\subset S^3$, and $K_0$ a core loop of the solid torus which is the exterior of $V$ in $S^3$.  
Set $K:=K_0\cup K_1$, and let $K(p/q,r/s)$ denotes the $3$-manifold obtained 
by the $p/q$-surgery on $K_0$ and the $r/s$-surgery on $K_1$. 
In this paper, $K(p/q,-)$ (resp. $K(-,r/s)$) denotes the $3$-manifold obtained 
by the $p/q$-surgery on $K_0$ (resp. the $r/s$-surgery on $K_1$) 
and removing  an open tubular neighborhood of $K_1$ (resp. $K_0$). 

\begin{proposition}
Suppose that $K(-,r/s)\cong D^2\times S^1$. 
Then $K(p/q,r/s)\cong L(pr-(n^2s)q,xq-yp)$, where $x$ and $y$ are coprime integers 
satisfying $y(n^2s)-xr=1$. 
\end{proposition}

\begin{proof}
Since we suppose that $K(-,r/s)\cong D^2\times S^1$, 
it follows from \cite[Lemma 3.3(ii)]{G} that the meridian of the new solid torus is 
given by the slope $r/(n^2s)$. Hence the conclusion immediately follows from\cite[Lemma 3]{BHW}. 
\end{proof}

\begin{figure}[tb]\begin{center}
{\unitlength=1cm
\begin{picture}(3.15,3.15)
 \put(0,0){\includegraphics[keepaspectratio]%height31.106mm
 {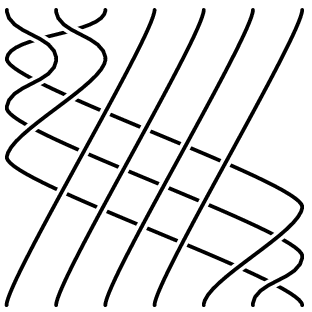}}
 \end{picture}}
 \caption{$\beta=W^{-1}_3  W^3_7$. }
\label{Fig2}
\end{center}\end{figure}

In the following of this section, let $\beta$ be the $7$-string braid and 
$K_1$ its closure in the solid torus $V$ illustrated in Figure \ref{Fig2}. 
We note that $K_1$ is denoted by $W^{-1}_3  W^3_7$ in \cite{B2}. 
It follows from \cite{B2} and \cite{BHW} that $K(-,18/1)\cong D^2\times S^1$.  
Since $K_1$ is a $7$-string braid, we see that $K(p/q , 18/1) \equiv L(18p-49q,19q-7p)$. 

\begin{proposition}\label{proposition4.2}
Let $K'_1$ be the image of $K_1$ in the lens space obtained by the $p/q$-surgery on $K_0$ with 
$p>0$.  Then there exists an integer $c>0$ such that $K'_1$ is hyperbolic in the lens space 
for any integer $q$ with $(p,q)=1$ and $|q| \ge c$. 
\end{proposition}

\begin{proof}
Since the number of strands of the braid $\beta$ is $7$, which is a prime, 
and the exponent sum of $\beta$ is $16$, 
which is not a multiple of $7-1=6$, 
it follows from \cite[Proposition 9.4]{Boyland} that the braid $\beta$ is pseudo-Anosov. 
Let $\rho$ be the pseudo-Anosov homeomorphism of a punctured disk obtained from the braid $\beta$. 
See \cite{Boyland} for the definition of pseudo-Anosov homeomorphisms for example. 
By the definition, there exits a measured foliation $\tau$ on the punctured disk, 
which is invariant for the pseudo-Anosov homeomorphism $\varphi$ corresponding to $\beta$. 
Let $E(K_1)$ be the exterior of $K_1$ in $V$, that is, 
$E(K_1) = \mathrm{cl} ( V - N (K_1) )$, where $N (K_1)$ denotes a tubular neighborhood of $K_1$. 
By regarding this $E(K_1)$ as 
the surface bundle over the circle with monodromy $\varphi$, 
we find an essential lamination $\mathcal{L}$ in the exterior 
as a suspension of $\tau$. 
See \cite{GO} for example. 

In the complement of $\mathcal{L}$, 
we have an annulus $A$ connecting 
from a leaf of $\mathcal{L}$ to the boundary $\partial N (K_1)$, 
which comes from 
the suspension of the arc on the punctured disk 
connecting a leaf of $\tau$ to a boundary circle. 
The boundary component of $A$ on $\partial N (K_1)$ 
determines a slope, which is so-called \textit{degeneracy slope} for $\mathcal{L}$. 
Denote it by $\gamma = u/v$ with $u>0$. 

It then follows from \cite[Theorem 2.5]{Wu} 
that $K'_1$ is hyperbolic if $\Delta (p/q,\gamma)=|pv-qu| \ge 3$. 

Since $p,u$ and $v$ are constant and $u>0$, if we take an integer $q$ with $q \le pv+3$, 
we see that $pv-qu \le pv-(pv+3)u =-3u \le -3$ and hence $K'_1$ is hyperbolic. 
\end{proof}

For a given lens space $L(p,q)$ and an integer $c>0$, we can always find an integer 
$q'_0$ with $q'_0>c$ and $L(p,q'_0) \equiv L(p,q)$,  because we, if necessary,  can replace 
$q$ by $q'_0:=q + pn$ $(n\in \mathbb{Z})$. 
Moreover, there is an infinite set of integers $Q=\{q'\ |\ q'>c,\ q'=q + pn (n\in \mathbb{Z})\}$. 
Then it follows from Proposition \ref{proposition4.2} that $K'_1$ is hyperbolic in the lens space $L(p,q') \equiv L(p,q)$ 
for any $q'\in Q$. 
Since $K(p/q' , 18/1) \equiv L(18p-49q',19q'-7p)$ and $L(p,q') \equiv L(p,q)$, we have: 

\begin{theorem}
For every $[ L ] \in \mathcal{L}$, 
there exists an infinite family $[L_i] \in \mathcal{L}$ such that $d_H ( [ L ] , [ L_i ] ) = 1$ for any $i\in \mathbb{N}$. 
\end{theorem}

Using arguments similar to the above, we also have: 

\begin{theorem}
For every $p \in \mathbb{N}$, 
there exist two pairs of coprime integers $(r,s)$ and $(r',s')$ such that 
$d_H ( [ L(r,s) ] , [ L(r',s') ] ) = 1$ and $| r -r' |=p$. 
\end{theorem}

\begin{proof}
We use the same link $K=K_0\cup K_1$ as above, and let $K'_1$ be again the image of $K_1$ 
in the lens space obtained by the $p/q$-surgery on $K_0$ with $p>0$. 
Then it follows from \cite{B2} that $K(-,19/1)\cong D^2\times S^1$. Also see \cite{BHW}. 
Hence we see that $K(p/q , 19/1) \equiv L( 19 p - 49 q , 18 q - 7 p)$.  
As mentioned above, $K(p/q , 18/1) \equiv L(18p-49q,19q-7p)$. 
This implies that the dual knot of $K'_1$ in $L(19p-49q,18q-7p)$ admits a Dehn surgery 
yielding $L(18p-49q,19q-7p)$.  
Hence we have 

\begin{center}
$d ( [L( 18 p - 49 q , 19 q - 7 p)]  , [L( 19 p - 49 q , 18 q - 7 p)] ) = 1$.
\end{center}

Moreover, retaking $q$ with an appropriate integer $q'$ with $q'>c$, we see that 

\begin{center}
$d_H ( [L( 18 p - 49 q' , 19 q' - 7 p)]  , [L( 19 p - 49 q' , 18 q' - 7 p)] ) = 1$.
\end{center}

Setting $(r,s)=(18p-49q',19q'-7p)$ and 
$(r',s')=(19p-49q',18q'-7p)$, we obtain the desired conclusion. 
\end{proof}

%%%%%%%%%%%%%%%%%%%%%%%%%%%%%%%%%%%%%%%%%

\section{Sample calculations}

Recall that 
if $d_H ( [ L ] , [ L' ] ) = 1$, then $ d_H ( [ L ] , [ L' ] )_\mathcal{L} = 1$ by definition; 
however, if $d_H ( [ L ] , [ L' ] ) = 2$, then $ d_H ( [ L ] , [ L' ] )_\mathcal{L} \ge 2$ in general. 
Consider the question: 
For which $[ L ] , [ L' ] \in \mathcal{L}$, $d_H ( [ L ] , [ L' ] ) = d_H ( [ L ] , [ L' ] )_\mathcal{L} = 2$. 
In this section,  we give some examples concerning this question. 

In the following, the link $K:=K_0\cup K_1$ illustrated in Figure \ref{Fig1} plays an important role. 
We note that $K$ is introduced by Yamada and is denoted by $k(3,5) \cup u$ in \cite{Yamada}. 

\begin{figure}[tb]\begin{center}
{\unitlength=1cm
\begin{picture}(5.6,4.6)
 \put(0,0){\includegraphics[keepaspectratio]%height46.015mm
 {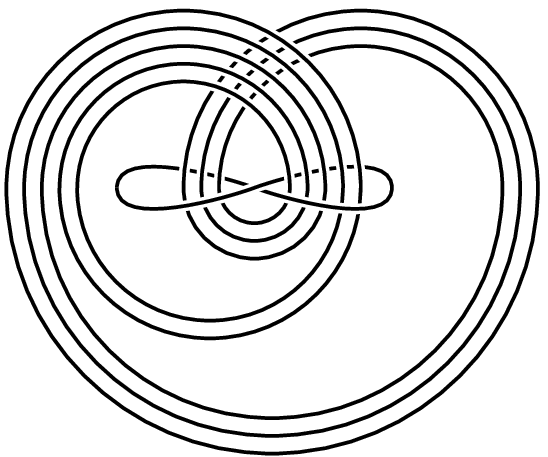}}
  \put(4.02,2.65){$K_0$}
  \put(5.5,2.65){$K_1$}
\end{picture}}
\caption{$K:=K_0\cup K_1$. }
\label{Fig1}
\end{center}\end{figure}

\subsection{\mbox{\boldmath $d_H ( [ S^3 ] , [ S^2 \times S^1 ] ) = d_H ( [ S^3 ] , [ S^2 \times S^1 ] )_\mathcal{L} = 2$}}

By an argument similar to that in \cite{Yamada}, we have $K(r/1,15/1)\equiv L(64-15r,23-5r)$. 
This implies that $K(0/1,15/1)\equiv L(64,23)$ and hence 
$d([S^2 \times S^1],[L(64,23)])=1$. Let $K'_1$ be the image of $K_1$ in $S^2 \times S^1$ 
which is obtained by the $0/1$-surgery on $K_0$. 
Then it is verified by using computer program SnapPea \cite{SnapPea} that 
$K'_1$ is hyperbolic in $S^2 \times S^1$. Hence $d_H ([S^2 \times S^1],[L(64,39)])=1$. 

On the other hand, we have $d_H ([S^3 ],[L(64,23)])=1$ as follows. 
Let $K''$ be the knot in $L(64,23)$ denoted by $K(L(64,23);19)$ (see the appendix for the definition). 
Then we see that $K''$ admits a Dehn surgery yielding $S^3$. Moreover, it follows from 
\cite[Theorem 1.3]{Saito2} that $K''$ is hyperbolic in $L(64,23)$ (see the appendix for detail). 

\subsection{\mbox{\boldmath $d_H ( [ S^3 ] , [ RP^3 ] ) = d_H ( [ S^3 ] , [ RP^3 ] )_\mathcal{L} = 2$}}

In the same way as above, 
we have that $K(2/1,15/1)\equiv L(34,13)$ and hence $d([RP^3],[L(34,13)])=1$. 
Let $K'_1$ be the image of $K_1$ in $RP^3$ 
which is obtained by the $2/1$-surgery on $K_0$. 

Again it is verified by using computer program SnapPea \cite{SnapPea} that 
$K'_1$ is hyperbolic in $RP^3$. Hence $d_H ([RP^3],[L(34,13)])=1$. 

On the other hand, we see $d_H ([S^3 ],[L(34,13)])=1$ as follows. 

Let $K''$ be the knot in $L(34,13)$ denoted by $K(L(34,13);9)$. 
Then we see that $K''$ admits a Dehn surgery yielding $S^3$. Moreover, it follows from 
\cite[Theorem 1.3]{Saito2} that $K''$ is hyperbolic in $L(34,13)$. 

%%%%%%%%%%%%%%%%%%%%%%%%%%%%%%%%%%%%

\section*{Acknowledgments}

The first author is partially supported by 
Grant-in-Aid for Young Scientists (B), No.20740039, 
Ministry of Education, Culture, Sports, Science and Technology, Japan. 
The second author is partially supported by 
JSPS Research Fellowships for Young Scientists.

The authors would like to thank 
Akio Kawauchi for teaching them about his imitation theory, 
for example, in \cite{Kawauchi}. 
They also thank Yukihiro Tsutsumi, Yuichi Yamada and Eiko Kin for useful discussion.

%%%%%%%%%%%%%%%%%%%%%%%%%%%%%%%%%%%%

\bibliographystyle{amsplain}

%%%%%%%%%%%%%%%%%%%%%%%%%%%%%%%%%%%%

\appendix
\section{Definition and properties of $K(L(p,q);u)$}
Recall the definition and the parametrization of lens spaces as follows. 
Let $V_1$ be a regular neighborhood of a trivial knot in $S^3$, 
$m$ a meridian of $V_1$ and $\ell$ a longitude of $V_1$ 
such that $\ell$ bounds a disk in $\mathrm{cl}(S^3\setminus V_1)$. We fix an orientation of $m$ and $\ell$ 
as illustrated in Figure \ref{FigA}. 
By attaching a solid torus $V_2$ to $V_1$ so that $\bar{m}$ is isotopic to 
a representative of  $p[\ell]+q[m]$ in $\partial V_1$, 
we obtain a lens space $L(p,q)$, where $p$ and $q$ are integers with $p>0$ and $(p,q)=1$, 
and $\bar{m}$ is a meridian of $V_2$. 
Then the intersection points of $m$ and $\bar{m}$ are labeled $P_0,\ldots ,P_{p-1}$ successively 
along the positive direction of $m$. Let $t^u_i$ $(i=1,2)$ be a simple arc in $D_i$ joining $P_0$ to $P_u$ 
$(u=1,2,\dots,p-1)$. 
Then the notation $K(L(p,q);u)$ denotes the knot $t^u_1\cup t^u_2$ in $L(p,q)$ 
(\textit{cf}. Figure \ref{FigA}). 

\begin{figure}[tb]\begin{center}
{\unitlength=1cm
\begin{picture}(12,5.7)
 \put(2.41,0){\includegraphics[keepaspectratio]%height55.19mm
 {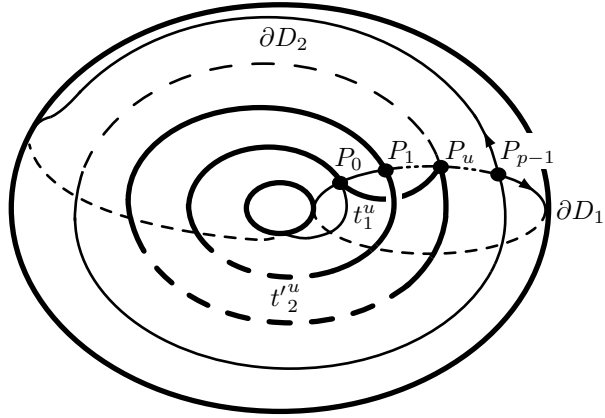}}
  \put(5.75,4.9){$\partial D_2$}
  \put(6.75,3.3){$P_0$}
  \put(7.45,3.4){$P_1$}
  \put(8.22,3.4){$P_u$}
  \put(8.95,3.4){$P_{p-1}$}
  \put(7,2.55){$t^u_1$}
  \put(9.7,2.6){$\partial D_1$}
  \put(5.9,1.45){${t'}^u_2$}
\end{picture}}
\caption{Here, ${t'}^u_2$ is a projection of $t^u_2$ on $\partial V_1$. }
\label{FigA}
\end{center}\end{figure}

We then prepare the following notations. 
Let $p$ and $q$ be integers with $p>0$ and $(p,q)=1$. 
Let $\{s_j\}_{1\le j\le p}$ be the finite sequence, which we call the \textit{basic sequence}, 
such that $0\le s_j< p$ and $s_j\equiv jq$ $(\bmod \ p)$. 
For an integer $u$ with $0<u<p$, $\Psi_{p,q} (u)$ denotes the integer 
satisfying $\Psi_{p,q} (u)\cdot q\equiv u$ $(\bmod \ p)$ and 
$\Phi_{p,q} (u)$ denotes the number of elements 
of the following set (possibly, the empty set):
\[
\{s_j\ |\ 1\le j< \Psi_{p,q} (u),\ s_j< u\}. 
\]
Also, $\widetilde{\Phi}_{p,q} (u)$ denotes the following: 
\[
\widetilde{\Phi}_{p,q} (u)=
\min \left\{
\begin{array}{l}
\Phi_{p,q} (u),\ \Phi_{p,q} (u)-\Psi_{p,q} (u)+p-u,\\
\ \ \Psi_{p,q} (u)-\Phi_{p,q} (u)-1,\ u-\Phi_{p,q} (u)-1
\end{array}
\right\}.
\]

\begin{figure}[tb]\begin{center}
{\unitlength=1cm
\begin{picture}(11.136,4.8)
 \put(0,0){\includegraphics[keepaspectratio]%height41.781mm
 {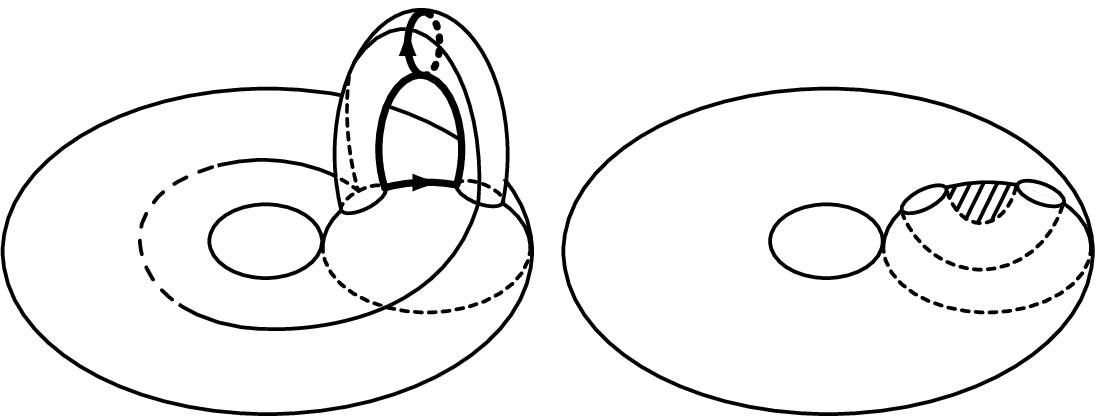}}
  \put(4.1,4.3){$m'$}
  \put(4.2,2.5){$\ell'$}
  \put(0,0){$V'_1$}
  \put(5.75,0){$V'_2$}
  \put(2.5,.5){$\partial D_2'$}
  \put(9.8,2.6){$D_2'$}
\end{picture}}
\caption{}
\label{FigB}
\end{center}\end{figure}

Set $V_1'=V_1\cup \eta(t_2^u;V_2)$, $V_2'=\mathrm{cl}(V_2\setminus \eta(t_2^u;V_2))$ and 
$S'=\partial V_1'=\partial V_2'$. Then $(V_1',V_2';S')$ is a genus two Heegaard splitting of $L(p,q)$. 
Let $D_2'\subset (D_2\cap V_2')$ be a meridian disk of $V_2'$ with $\partial D_2' \supset ({t'}_2^u \cap S')$. 
Let $m'$ be a meridian of $K=t^u_1\cup t^u_2$ in the annulus $S'\cap \partial \eta(t_2^u;V_2)$. 
Let $\ell'$ be an essential loop in $S'$ which is a union of ${t'}^u_1 \cap S'$ and 
an essential arc in the annulus $S'\cap \partial \eta(t_2^u;V_2)$ disjoint from $\partial D_2'$ 

Let $m^{\ast}$ be a meridian of $K$ in $\partial \eta(K;V_1')$ 
and $\ell^{\ast}$ a longitude of $\partial \eta(K;V_1')$ such that $\ell'\cup \ell^{\ast}$ bounds 
an annulus in $\mathrm{cl}(V_1'\setminus \eta(K;V_1'))$. The loops $m^{\ast}$ and $\ell^{\ast}$ are 
oriented as illustrated in Figure \ref{FigB}. Then $\{[m^{\ast}],[\ell^{\ast}]\}$ is  
a basis of $H_1(\partial \eta(K;V_1');\mathbb{Z})$. Let $V_1''$ be a genus two handlebody obtained 
from $\mathrm{cl}(V_1'\setminus \eta(K;V_1'))$ by attaching a solid torus $\bar{V}$ so that the 
boundary of a meridian disk $\bar{D}$ of $\bar{V}$ is identified with a loop represented by 
$r[m^{\ast}]+s[\ell^{\ast}]$. Set $M'=V_1''\cup_{S'} V_2'$. 
Then we say that $M'$ is obtained by \textit{$(r/s)^{\ast}$-surgery} on $K$. 
We note that $(r/s)^{\ast}$-surgery is longitudinal if and only if $r/s$ is an integer. 

\subsection{The fundamental group}
Since $K(L(p,q);u)$ is a $(1,1)$-knot in $L(p,q)$, particularly is a so-called $1$-bridge braid, 
we can easily obtain a presentation of the fundamental group of a surgered manifold as follows. 

\begin{proposition}[{\cite[Theorem 5.1]{Saito3}}]
Set $K=K(L(p,q);u)$ and let $\{s_j\}_{1\le j\le p}$ be the 
basic sequence for $(p,q)$. Let $N'$ be the $3$-manifold obtained 
by $r^{\ast}$-surgery on $K$, where $r$ be an integer. Then 
we have:

\begin{center}
$\pi_1 (N')\cong \Bigg\langle a,b\ 
\Bigg|\ 
\displaystyle\prod_{j=1}^{\Psi_{p,q} (u)} W_1(j) =1,\ 
\displaystyle\prod_{j=1}^{p} W_2(j) =1 \Bigg\rangle$,  
\end{center}

where 

\begin{center}
$W_1 (j)=\left\{
     \begin{array}{cl}
     a    & \mathrm{if}\ \, s_j >u\\
     ab^r & \mathrm{if}\ \, s_j=u\\
     ab   & \mathrm{otherwise}
     \end{array}
\right.$ 
and
$W_2 (j)=\left\{
     \begin{array}{cl}
     a    & \mathrm{if}\ \, s_j \ge u\\
     ab   & \mathrm{otherwise}
     \end{array}
\right.$. 
\end{center}
\end{proposition}

\subsection{Hyperbolicity}
Though $K=K(L(p,q);u)$ admits several representation (\textit{cf}. \cite[Proposition 4.5]{Saito2}), 
it is proven in that $\widetilde{\Phi}_{p,q} (u)$ is an invariant for $K$ if $K$ 
admits a longitudinal surgery yielding $S^3$ (\textit{cf}. \cite[Corollary 4.6]{Saito2}). 
Hence when  $K$ admits a longitudinal surgery yielding $S^3$, $\widetilde{\Phi}_{p,q}(u)$ 
is denoted by $\Phi (K)$.  Moreover, we have a necessary and sufficient condition for such knots 
to be hyperbolic. 

\begin{proposition}[{\cite[Theorem 1.3]{Saito2}}]
Set $K=K(L(p,q);u)$. Suppose that $K$ admits a longitudinal surgery yielding $S^3$. 
Then we have the following: 

\begin{enumerate}
\item $\Phi (K)=0$ if and only if $K$ is a torus knot. 
\item $\Phi (K)=1$ if and only if $K$ contains an essential torus in its exterior. 
\item $\Phi (K)\geq 2$ if and only if $K$ is a hyperbolic knot. 
\end{enumerate}
\end{proposition}

\subsection{Example}
Set $K=K(L(64,23);19)$. The basic sequence for $(64,23)$ is: 

\noindent
$\{s_j\}_{1\leq j\leq 64}:\ 23,46,5,28,51,10,33,56,15,38,61,20,43,2,25,48,7,30,53,12,35,58,\\
\hspace{2.045cm} 17,40,63,22,45,4,27,50,9,32,55,14,37,60,19,42,1,24,47,6,29,52,\\
\hspace{2.045cm} 11,34,57,16,39,62,21,44,3,26,49,8,31,54,13,36,59,18,41,0.$

Let $N'$ be the $3$-manifold obtained by $1^{\ast}$-surgery on $K$. Then we have:

\[
\pi_1 (N') 
\cong \bigg\langle a,b\, 
\left|
\begin{array}{l}
(a^3b)^3a^5b(a^3b)^3a^5b(a^3b)^3,\\
(a^3b)^3a^5b(a^3b)^3a^5b(a^3b)^2a^5b(a^3b)^3a^5b(a^3b)^3a^2b
\end{array}
\right\rangle.
\]

Repeating word reduction, we see that $\pi_1 (N')$ is trivial. This implies that 
$N'\cong S^3$ since Geometrization Conjecture is true \cite{P1,P2,P3}. 
Moreover, $K$ is hyperbolic since $\Phi(K)=8$.

%%%%%%%%%%%%%%%%%%%%%%%%%%%%%%%%%%%%

\end{document}